\documentclass[12pt, reqno]{amsart}
\usepackage{amsmath, amsthm, amscd, amsfonts, amssymb, graphicx, color}
\usepackage[bookmarksnumbered, colorlinks, plainpages]{hyperref}
\hypersetup{colorlinks=true,linkcolor=red, anchorcolor=green, citecolor=cyan, urlcolor=red, filecolor=magenta, pdftoolbar=true}

\newtheorem{theorem}{Theorem}[section]
\newtheorem{lemma}[theorem]{Lemma}
\newtheorem{proposition}[theorem]{Proposition}

\theoremstyle{definition}

\newtheorem{example}[theorem]{Example}

\theoremstyle{remark}

\numberwithin{equation}{section}

\begin{document}

\setcounter{page}{1}

\title[Which Topologies induced by order convergences ]{Which Topologies induced by order convergences }

\author[ K. Haghnejad Azar]{ Kazem Haghnejad Azar$^{*}$}

\address{Department of Mathematics, University of Mohaghegh Ardabili, Ardabil, Iran.
\newline
}
\email{\textcolor[rgb]{0.00,0.00,0.84}{haghnejad@uma.ac.ir}}


\let\thefootnote\relax\footnote{Copyright 2016 by the Tusi Mathematical Research Group.}

\subjclass[2010]{Primary 46B42; Secondary 47B60.}

\keywords{Riesz space, unbounded order convergence, $uo$-continuous operator, $u\sigma o$-continuous operator.}

\date{Received: xxxxxx; Revised: yyyyyy; Accepted: zzzzzz.
\newline \indent $^{*}$Corresponding author}

\begin{abstract}
In this paper, we will study on some topologies induced by order convergences in a vector lattice. We will investigate the relationships of them. 
\end{abstract} \maketitle

\section{Introduction }
Recall that a net $(x_\alpha)_{\alpha\in \mathcal{A}}$ in a Riesz space  $E$ is  {\it  order convergent}  to  $x\in E$, denoted by $x _\alpha \xrightarrow{o}x$ whenever there exists another net $(y_\beta)_{\beta\in \mathcal{B}}$  in $E$ such that $y_\beta \downarrow 0$ and that for every $\beta\in \mathcal{B}$, there exists $\alpha_0\in \mathcal{A}$ such that  $|x_\alpha -x| \leq y_\beta$ for all $\alpha\geq \alpha_0$. If there exists a net $(y_\alpha)_{\alpha\in \mathcal{A}}$ (with the same index set)  in a Riesz space $E$  such that $y_\alpha\downarrow 0$ and  $|x_\alpha -x| \leq y_\alpha$ for each $\alpha\in \mathcal{A}$, then $x _\alpha \xrightarrow{o}x$. Conversely, if $E$ is a Dedekind complete Riesz space and  $(x_\alpha)_{\alpha\in \mathcal{A}}$ is order bounded, then $x _\alpha \xrightarrow{o}x$ in $E$ implies that there exists a net  $(y_\alpha)_{\alpha\in \mathcal{A}}$ (with the same index set) such that  $y_\alpha\downarrow 0$ and  $|x_\alpha -x| \leq y_\alpha$ for each $\alpha\in \mathcal{A}$. For sequences in a Riesz space $E$,  $x _n \xrightarrow{o}x$ if and only if there exists a sequence $(y_n)$ such that $y_n\downarrow 0$ and $|x_n -x| \leq y_n$ for each $n\in \mathbb{N}$ (cf.  \cite[P.17 and P.18]{1}).  
 
We adopt \cite{2} as standard reference for basic notions on Riesz spaces and Banach lattices. Recall that a real vector space $E$ (with elements $x$,$y$,...) is called an ordered vector space if $E$ is partially ordered in such a manner that the vector space structure and order structure are compatible, that is to say, $x\leq y$ implies $x+z\leq y+z$ for every $z\in E$ and $x\geq y$ implies $\alpha x\geq \alpha y$ for every $\alpha \geq 0$ in $\mathbb{R}$.  A  Riesz space $E$ is an order vector space in which sup$(x,y)$ ( it is customary to write sometimes $x\vee y$ instead of  sup$(x,y)$ and $x\wedge y$ instead of inf$(x,y)$ ) exists for every $x,y\in E$. Let $E$ be a Riesz space, for each $x,y \in E$ with $x \leq y$, the set $[x,y]=\{z \in E : x\leq z\leq y\}$ is called an order interval.  A subset of $E$ is said to be order bounded if it is included in some order interval. A Riesz space is said to be Dedekind complete (resp. $\sigma$-Dedekind complete) if  every order bounded above subset (resp. countable subset)  has a supremum.  A subset $A$ of a Riesz space $E$ is said to be solid if it follows from $|y|\leq |x|$ whit $x\in A$ and $y\in E$ that $y\in A$ . An order ideal of $E$ is a solid subspace. A band of $E$ is an order closed order ideal.  A Banach lattice $E$ is a Banach space $(E, \|.\|)$ such that $E$ is a Riesz space  and its norm satisfies the following property: for each $x,y \in{E}$ such that $ |x| \leq |y|$, we have $\|x\|\leq \|y\|$.  A Banach lattice $E$ has order continuous norm if $\| x_\alpha\|\rightarrow 0$ for every decreasing net $(x_\alpha)_\alpha$ with $\inf_\alpha x_\alpha=0$. A vector $x>0$ in a Riesz space $E$ is called an atom if $E_x=\{y\in E : \exists \lambda >0, \  |y| \leq  \lambda x \}$, the ideal generated by $x$,  is one-dimensional if and only if $u, v\in [0,x]$ with $u\wedge v=0$ implies $u=0$ or $v=0$. A Riesz space $E$ is said to be atomic if the linear span of all atoms is order dense in $E$ if and only if it is the band generated by its atoms. For example $c$, $c_0$, $\ell_p$($1\leq p\leq \infty)$ are atomic Banach lattices and $C[0,1]$, $L_1[0,1]$ are atomless Banach lattices.  Let $E$, $F$ be Riesz spaces. An operator $T:E\rightarrow F$ is said to be order bounded if it maps each order bounded subset of $E$ into order bounded subset of $F$. The collection of all order bounded operators from a Riesz space $E$ into a Riesz space $F$ will be denoted by $\mathcal{L}_b(E,F)$. The collection of all order bounded linear functionals on a Riesz space $E$ will be denoted by $E^\sim$,  that is $E^\sim=\mathcal{L}_b(E,\mathbb{R})$ . A functional on a Riesz space is order continuous (resp. $\sigma$-order continuous) if it maps order null nets (resp. sequences) to order null nets (resp. sequences). The collection of all order continuous (resp. $\sigma$-order continuous) linear functionals on a Riesz space $E$ will be denoted by $E_n^\sim$ (resp. $E_c^\sim$). For unexplained terminology and facts on  Banach lattices and positive operators, we refer the reader to the excellent book of \cite{2}.

\section{Order Topology}

Let $E$ be a vector lattice. A subset $A$ of a $E$ is said to be quasi-order closed whenever for every  $(x _\alpha)\subseteq A$
with  $x _\alpha\uparrow x$ or  $x _\alpha\downarrow x$ implies $x\in A$. We observe that a solid subset $A\subseteq E$ is a quasi-order closed if and only if $A$ is order closed. $\theta\subseteq E$ is called order open if and only if $E\setminus \theta$ is quasi-order closed. Now consider the following topologies:
\begin{enumerate}
\item First topology is called quasi-order topology which we define as follows.
\begin{equation*}
\tau_o=\{\theta\subseteq E:~E\setminus \theta~\text{is~quasi-order~closed}\}
\end{equation*}
It is clear,  $\tau_o$ is a topology for $E$.
\item Assume that $\tau_e$ be a topology for $E$ with following basis $$\{(a,b):~a,b\in E ~\text{and}~ a<b\}.$$ We call this topology as order topology.
\end{enumerate}
 In the following proposition, we show that $(E,\tau_o)$ and $(E,\tau_e)$  are both vector topologies.

\begin{proposition}
Let $E$ be a Dedekind complete vector lattice. Then  $\tau_o$ and $\tau_e$ both are vector topology.
\end{proposition}
\begin{proof}
Obvious that  $\tau_e$ is a vector topology. We only show that $\tau_o$ is vector topology.
First, we prove that the operation $x\rightarrow tx$ for each $t\in R$ is continuous. Let $\theta\subset E$ be an order open subset of $E$, then  we must show that $t\theta$ is an order open subset of $E$ for each $t\in R$. Since $\theta$ is order open, it follows that $\theta^c=F$ is quasi-order closed. Put $(t\theta)^c=G$ and $(x_\alpha)\subseteq G$ with $x_\alpha\uparrow x$. Then we have $x_\alpha \notin t\theta$ iff $t^{-1}x_\alpha\notin\theta$ iff $t^{-1}x_\alpha\in F$ for each $\alpha$ and since $t^{-1}x_\alpha\uparrow t^{-1}x$,   follows that $t^{-1}x\in F$,            implies that   $x\in tF$.    Then we have $t^{-1}x\in F$ iff  $ t^{-1}x \notin \theta$ iff $x\notin \theta$ iff $x\in G$, which follows that $G$ is quasi order closed, and so $t\theta$ is an order open subset of $E$. \\
Now we show that the operation $(x,y)\rightarrow x+y$ is continuous. Set $\theta_1$ and $\theta_2$ order open subsets of $E$, we show that $\theta_1+\theta_2$ is an order open subset of $E$. Let $a\in \theta_1$. First we prove that $a+\theta_2$ is an order open subset of $E$.  Put $\theta_2^c=F$ and $(a+\theta_2)^c=G$.  We show that  $G$ is quasi-order closed. Let $(x_\alpha)\subseteq G$ and $x_\alpha\uparrow x$ in $G$. Then we have $x_\alpha\in G$       iff  $x_\alpha\notin (a+\theta_2)$                         iff   $(x_\alpha-a)\notin \theta_2$.  Since $(x_\alpha-a)\uparrow (x-a)$, follows that  $(x-a)\in F$, and so $(x-a)\notin \theta_2$ iff $x\notin (a+\theta_2)$       iff   $x\in G$.  Thus $G$ is quasi-order closed, and so $a+\theta_2$ is an order open subset of $E$. Now by $\theta_1+\theta_2= \bigcup_{a\in \theta_1}(a+ \theta_2)$, the proof follows.    

\end{proof}

\begin{lemma} \label{2.1}
Let $E$ be a Dedekind complete vector lattice and  $\tau_o$ be a order topology for $E$. Then for each $c\in E$ and neighborhood $U_c$ of $c$, there are $a,b\in E$ such that $c\in (a,b)\subset U_c$.
\end{lemma}

\begin{proof}
let $c\in E$ and $U_c$ be an neighbourhood of $c$ in order topology. First we show that there is $a\in E$ such that $(a,c)\subset U_c$. By contradiction, let $(a,c)\cap U_c^c\neq \emptyset$. Then for each $a<c$ there is $c_a\in (a,c)\cap U_c^c$. It follows that $$\sup\{c_a:~c_a\in (a,c)\cap U_c^c\}=c.$$ For each $a<b<c$, we can set $c_a<c_b$. It follows that  for each $a<c$, there exists $c_{\alpha (a)}\in {(a,c)\cap U_c^c}$ with $c_{\alpha (a)}\uparrow c$.  It follows that $c\in U_c^c$, which is not possible. Thus there is $a<x$ such that $(a,c)\subset U_c$. In the similar way there is a $c<b$ such that $(c,b)\subset U_c$ and proof follows.
\end{proof}

The preceding lemma shows that $\tau_o\subseteq \tau_e$, but as following example, in general two topologies not coincide. 
\begin{example}\label{e1}
Consider $E=\ell^\infty$ and $e_1=(1,0,0,0...)$. Then $(-e_1,e_1)$ is member of $\tau_e$, but  is not belong to $\tau_o$. Consider $x_n\in \ell^\infty$ which  first $n$ terms are zero and others are 1. Obviously $x_n\downarrow 0$, but $x_n\notin (-e_1,e_1)$ for each  $n$.  This example shows that the sequence $(x_n)$ is order convergent to zero, but is not topological convergence to zero. On the other hand, since $(-e_1,e_1)\notin \tau_o$, two topologies not coincide.
\end{example}

\begin{theorem}\label{t1}
Let $E$ be a Dedekind complete vector lattice with topology $\tau_e$ and  $(x_\alpha)\subset E$. If  $x_\alpha\xrightarrow{\tau_e}0$, then  $(x_\alpha)$ is order convergence to zero.
\end{theorem}
\begin{proof}
Assume that $a,b\in E$ with $x\in (a,b)\subseteq E$.  Since $x_\alpha\xrightarrow{\tau_e}x$, there exists $\alpha_{(a,b)}$ such that  $x_\alpha\in (a,b)$ for each  $\alpha\geq \alpha_{(a,b)}$. Put $y_{\alpha_{(a,b)}}= b-a$. On the other hands,  $(\alpha_{(a,b)})_{x\in (a,b)}$ is a directed set with the following order relation 
$$\alpha_{(a,b)}\leq\alpha_{(c,d)}~\text{iff}~ (c,d)\subseteq (a,b).$$
It follows that
\begin{equation*}
\vert x_\alpha -x\vert= (x_\alpha\vee x)- (x_\alpha\wedge x)\leq b-a=y_{\alpha_{(a,b)}}\downarrow 0.
\end{equation*}
Thus $x_\alpha\xrightarrow{o}x$.
\end{proof}
By Example \ref{e1}, the converse of Theorem \ref{t1} in general not holds.

\begin{proposition}
Let $E$ be a Dedekind complete vector lattice and  $\tau_o$ be an order topology for $E$. If $B$ is an ideal and quasi-order closed subset of $E$, then  $B$ is a band in $E$.
\end{proposition}
\begin{proof}
 Let  $(x_\alpha)\subseteq B$ and $x_\alpha\xrightarrow{o}x$, we show that $x\in B$.
 Obversely $\sup\{x_\alpha \wedge x\}=x$. Set $y_\beta =(\bigvee_{\alpha \leqslant\beta} x_\alpha)\wedge x$, then $y_\beta\uparrow x$. Since $(y_\beta)\subseteq B$ and $B$ is quasi-order closed,  follows that $x\in B$ and  the result follows.
\end{proof}


\bibliographystyle{amsplain}

\end{document}